\documentclass{amsart}
\usepackage{amssymb}
\usepackage{amscd}
\usepackage{graphicx}
\usepackage{amsmath}
\usepackage{setspace}
\usepackage[all]{xy}
\usepackage{color}
\usepackage{enumerate}

\newtheorem{theorem}{Theorem}[section]

\newtheorem{corollary}[theorem]{Corollary}

\newtheorem{lemma}[theorem]{Lemma}

\newtheorem{proposition}[theorem]{Proposition}

\theoremstyle{remark}

\newtheorem{example}[theorem]{Example}

\newcommand{\CB}{\mathcal{B}}

\newcommand{\CALD}{\mathcal{D}}
\newcommand{\CF}{\mathcal{F}}

\newcommand{\CI}{\mathcal{I}}

\newcommand{\CP}{\mathcal{P}}

\newcommand{\CT}{\mathcal{T}}

\newcommand{\Hom}{\mathrm{Hom}}
\newcommand{\Ima}{\mathrm{Im}}

\newcommand{\Ker}{\mathrm{Ker}}
\newcommand{\Coker}{\mathrm{Coker}}
\newcommand{\Ext}{\mathrm{Ext}}

\newcommand{\Gen}{\mathrm{Gen}}
\newcommand{\gen}{\mathrm{gen}}
\newcommand{\Pres}{\mathrm{Pres}}

\newcommand{\Cogen}{\mathrm{Cogen}}

\newcommand{\Add}{\mathrm{ Add}}

\newcommand{\Modr}{\mathrm{ Mod}\text{-}}

\newcommand{\pushoutcorner}[1][dr]{\save*!/#1-1.2pc/#1:(-1,1)@^{|-}\restore}

\begin{document}


\title{Torsion classes generated by silting modules}

\author{Simion Breaz}

\address{Simion Breaz: "Babe\c s-Bolyai" University, Faculty of Mathematics and Computer Science, Str. Mihail Kog\u alniceanu 1, 400084, Cluj-Napoca, Romania}

\email{bodo@math.ubbcluj.ro}

\author{Jan \v Zemli\v cka}
\address{Jan \v Zemli\v cka: Department of Algebra, Charles University in Prague, Faculty of Mathematics and Physics Sokolovsk\' a 83, 186 75 Praha 8, Czech Republic}
\email{zemlicka@karlin.mff.cuni.cz}



\begin{abstract}
We study the classes of modules which are generated by a silting module. In the case of either hereditary or perfect rings it is proved that these are exactly the torsion $\mathcal{T}$ such that the regular module has a special $\mathcal{T}$-preenvelope. In particular, every torsion enveloping class in $\Modr R$ are of the form $\Gen(T)$ for a minimal silting module $T$. For the dual case we obtain for general rings that the covering torsion-free classes of modules  are exactly the classes of the form $\Cogen(T)$,
where $T$ is a cosilting module. 
\end{abstract}

\maketitle

\section{Introduction}

The study of torsion theories which are (co)generated by some special modules is useful since in many cases these torsion theories can be characterized by some intrinsic properties. For instance, it was proved in \cite[Proposition 1.1 and Section 2]{AIR} that in the case of finitely generated modules over artin algebras the classes of the form $\gen(T)$ (i.e. epimorphic images of finite direct sums of copies of $T$) induced by a $\tau$-tilting module $T$ coincide with the torsion classes which are enveloping. We refer to \cite[Section 5]{Aus-adv} for similar characterizations in the (co)tilting cases. 

The notion of a silting module was introduced in \cite{AMV:2015} in order to extend the $\tau$-tilting theory, developed in \cite{AIR} and \cite{Jasso} for finitely generated modules over artin algebras, to infinitely generated modules. The dual notion, i.e. cosilting modules, was studied in \cite{Br-Po}.
As in the case of  (co)tilting modules, see \cite{AngeleriTonoloTrlifaj_2001}, a natural question is to ask for characterizations of torsion classes which are generated by some infinitely generated modules which have similar properties as $\tau$-tilting modules. 
Since module $T\in\Modr R$ is (co)silting if the class of all $T$-(co)generated modules is a torsion (torsion-free) class of a special form,
a natural approach to extend the question is to consider torsion classes generated by silting modules.
The dual notion, i.e. cosilting modules, was studied in \cite{Br-Po}.

We recall that silting modules are in one-to-one correspondence with silting objects in the derived category of $\Modr R$ which can be represented by complexes of the form $0\to P_{-1}\to P_0\to 0$ with $P_{-1}$ and $P_0$ projectives.  Therefore, they are also in correspondence with 
important concepts  as (co-)$t$-structures or simply-minded collections of objects (see \cite{Ko_Ya:2014} and \cite{AMV:2015}).
It was proved recently that for some classes of rings (e.g. hereditary or commutative ring), they can be parametrized by universal localizations, \cite{Ma-Sto:2015}, Gabriel topologies of finite type, \cite{An-Hr}, or wide subcategories of finitely presented modules \cite{AMV2}.
For other correspondences and constructions we refer to \cite{NSZ} and \cite{Ps-Vi}. For various correspondences in the cosilting case, we refer to \cite{ZW} and \cite{ZW2}. Moreover, the 0-th homologies of compact silting complexes of the above form appear naturally as generators for torsion theories $(\CT,\CF)$ in $\Modr R$ such that the heart of the associated $t$-structure is equivalent to a module category, \cite{AKM}, \cite{ma-to}. For some more general discussions, we refer to \cite{Par-Sa}. The complexity of the transfer from the finitely generated case to infinitely generated modules is described in \cite{trl-tau}.     

In this paper we provide a general characterization (Proposition \ref{silting-classes}) for silting classes, as torsion classes which are generated via some special pushout constructions. In the case when
$R$ is right perfect (Theorem \ref{perfect-silting}), respectively right hereditary (Theorem \ref{hereditary-silting}) it leads to  characterizations which can be viewed as extensions of the corresponding result for tilting classes, \cite[Theorem 2.1]{AngeleriTonoloTrlifaj_2001}. 
In particular, every enveloping torsion class of modules over a perfect ring or over a hereditary ring is generated by a silting module 
(Corollary \ref{env-her}). The case of perfect rings extends the corresponding results proved for finitely generated modules over an artin algebra in \cite{sma} and \cite[Theorem 2.7]{AIR}. 

The last section of the paper is devoted to the dual setting, namely, we consider torsion-free classes which are of the form $\Cogen(T)$, where $T$ is a cosilting module. 
Since injective modules form an enveloping class over a general ring, we obtain using dual tools that for \textsl{every} ring $R$ torsion-free covering classes in $\Modr R$ are exactly the classes which are cogenerated by cosilting modules (Theorem \ref{cosilting-hereditary}). 

In this paper $R$ is a unital ring, and  $\Modr R$ will denote the category of all right $R$-modules.  If $T$ is an $R$-module then $\Gen(T)$ (respectively $\Cogen(T)$)  denotes the the closure to isomorphisms of the class of all quotients (submodules) of direct sums (products) of copies of $T$.

\section{Silting classes}

If $\CP$ is the class of all projective modules in $\Modr R$ and $\CP^\to$ will denote the class of all homomorphisms 
$\sigma:P_{-1} \to P_{0}$ with $P_{-1},P_0\in\CP$. 

For every homomorphism $\sigma:P_{-1} \to P_{0}$  from $\CP^\to$  we can associate to $\sigma$ the class  
$$\CALD_{\sigma} = \{ X\in\Modr R \mid \Hom_{R}(\sigma,X) \text{ is an epimorphism}\}.$$
If $T$ is a right $R$-module then $\Gen(T)$ denotes the class of all epimorphic images of direct sums of copies of $T$.

If $\CT$ is a class of modules, we will use the following classes:
\begin{itemize}
\item ${^\circ\CT}=\{X\in\Modr R\mid \Hom(X,T)=0 \textrm{ for all } T\in\CT\}$,
\item$\CT^\circ=\{X\in\Modr R\mid \Hom(T,X)=0 \textrm{ for all } T\in\CT\}$,
\item ${^\perp\CT}=\{X\in\Modr R\mid \Ext^1(X,T)=0 \textrm{ for all } T\in\CT\}$,
\item $\CT^\perp=\{X\in\Modr R\mid \Ext^1(T,X)=0 \textrm{ for all } T\in\CT\}$,
\item ${^\square \CT}=\{\alpha\in\CP^\to\mid \CT\subseteq \CALD_\alpha\}$, and
\item ${^\diamond\CT}=\{\Coker(\alpha)\mid \alpha\in {^\square\CT}\}$.
\end{itemize}

Recall from \cite{AMV:2015} that a module $T$ is \textsl{partial silting} if there exists a projective presentation $$P_{-1}\overset{\sigma}\to P_0\to T\to 0$$ such that 
$\CALD_\sigma$ is a torsion class and $T\in \CALD_\sigma$. Then $\Gen(T)\subseteq \CALD_\sigma\subseteq {T^\perp}$ and $(\Gen(T), T^{\circ})$ is a torsion pair.
If  $\CALD_\sigma=\Gen(T)$ then $T$ is called a \textsl{silting module}.   

Let $\CT$ be a class of modules. Then a homomorphism $\epsilon:X\to T$ with $T\in \CT$ is a \textsl{$\CT$-preenvelope} if $\Hom(\epsilon, T')$ is surjective for all $T'\in \CT$, i.e. all homomorphisms $X\to T'$ with $T'\in \CT$ factorize through $\epsilon$. The $\CT$-preenvelope $\epsilon$ is a \textsl{$\CT$-envelope} if every homomorphism $\alpha:T\to T$ with the property $\epsilon=\alpha\epsilon$ has to be an isomorphism. If all modules $X\in\Modr R$ have a $\CT$-preenvelope (envelope) then $\CT$ is \textsl{preenveloping} (resp. \textsl{enveloping}). A preenveloping class $\CT$ is \textsl{special} if for every $X\in\Modr R$ we can find a $\CT$-preenvelope $\epsilon$ which is monic and $\Coker(\epsilon)\in {^\perp \CT}$. The corresponding dual notions are that of (special) precover/precovering and cover/covering, respectively.  

Tilting classes, i.e. the torsion classes of the form $\Gen(T)$ with $T$ a tilting module, can be characterized by the fact that they are exactly the special preenveloping
torsion classes in $\Modr R$ (cf. \cite[Section 5]{Aus-adv}, \cite[Theorem 2.1]{AngeleriTonoloTrlifaj_2001}). 
We refer to \cite{Fu-et-al}
for a recent study of this kind of special preenveloping situation which involves homomorphisms instead of objects. Even 
the orthogonality used in this paper does not cover the (co)silting case (cf. \cite[Remark 39]{Br-Mo}), we can characterize torsion classes of the form $\Gen(T)$ with $T$ a silting module by the existence of preenvelopes with some special properties.

\begin{proposition}\label{silting-classes}
The following are equivalent for a class $\CT$ of $R$-modules:
\begin{enumerate}[{\rm (1)}]
\item There exists a silting module $T$ such that $\CT=\Gen(T)$;

\item \begin{enumerate}[{\rm (a)}]
\item $\CT$ is a torsion class,
\item there exists a $\CT$-preenvelope $\epsilon:R\to M$ 
which can be obtained as a pushout 
\[\xymatrix{
L_{-1} \ar[r]^\rho\ar[d]^\delta \pushoutcorner & L_0\ar[r] \ar[d] & K \ar[r]\ar@{=}[d] & 0 \\
R \ar[r]^\epsilon & M\ar[r]  & K \ar[r] & 0 
}\]
such that $\rho\in{^\square \CT}$.
\end{enumerate}
\item 
\begin{enumerate}[{\rm (a)}]
\item $\CT$ is a torsion class,
\item for every $R$-module $X$ there exists an $\CT$-preenvelope $\epsilon:X\to T_0$ 
which can be obtained as a pushout 
\[\xymatrix{
L_{-1} \ar[r]^\rho\ar[d] \pushoutcorner& L_0\ar[r] \ar[d] & T_1 \ar[r]\ar@{=}[d] & 0 \\
X \ar[r]^\epsilon & T_0\ar[r]  & T_1 \ar[r] & 0 
}\]
such that $\rho\in{^\square \CT}$.
\end{enumerate}
\end{enumerate}

If we have a diagram as in {\rm (2)} then  $M\oplus K$ is a silting module and $\CT=\Gen(T)=\CALD_\rho$.
\end{proposition}

\begin{proof}
(1)$\Rightarrow$(3) Let $\sigma:P_{-1}\to P_0$ be a homomorphism from $\CP^\to$ such that $T=\Coker(\sigma)$, and $T$ is silting with respect to $\sigma$. Hence $\CT=\CALD_\sigma$. 

For every module $X$ we consider the canonical homomorphism $\delta:P_{-1}^{(I)}\to X$, where $I=\Hom_R(P_{-1},X)$, and we construct the pushout diagram 
\[\xymatrix{
P_{-1}^{(I)} \ar[r]^{\sigma^{(I)}}\ar[d]^{\delta} \pushoutcorner& P_{0}^{(I)}\ar[r] \ar[d]^{\delta_0} & T^{(I)} \ar[r]\ar@{=}[d] & 0 \\
X \ar[r]^\epsilon & T_0\ar[r]^\nu  & T^{(I)} \ar[r] & 0 
}\] 
Then, as in the proof of \cite[Theorem 3.12]{AMV:2015} we obtain that $T_0\in \CALD_\sigma=\CT$. 

Moreover, for every $Y\in \CT$ and every homomorphism $\alpha:X\to Y$ there exists $\beta:P_{0}^{(I)}\to Y$ such that $\delta\alpha=\beta\sigma^{(I)}$. By the pushout universal property there exists $\gamma:T_0\to Y$ such that $\alpha=\gamma\epsilon$, hence $\epsilon$ is a $\CT$-preenveloping map.
Since $\sigma^{(I)}\in {^\square \CT}$, the proof is complete.

(3)$\Rightarrow$(2) is obvious.

(2)$\Rightarrow$(1) If $X\in \CALD_\rho$ then every homomorphism $R\to X$ can be lifted to a homomorphism $M\to X$. Therefore every element of $X$ is in the image of a homomorphism $M\to X$, hence $X\in\Gen(M)$. It follows that $\CALD_\rho\subseteq \Gen(M)=\CT$. But $\CT\subseteq \CALD_\rho$ since $\rho\in {^\square \CT}$, and it follows that $\CALD_\rho=\CT$ is a torsion class. Moreover, $K\in \CT=\CALD_\rho$, hence $K$ is partial silting with respect to $\rho$. By the proof of \cite[Theorem 3.12]{AMV:2015} it follows that $T=M\oplus K$ is a silting module $T$ 
with respect to the projective resolution $$\gamma\oplus \rho: L_{-1}\oplus L_{-1}\to (L_0\oplus R)\oplus L_0,$$ where $L_{-1}\overset{\gamma}\to L_0\oplus R\to M$ is the canonical exact sequence induced by $\delta$ and $\rho$, and that  
$\Gen(T)=\CALD_\rho=\CT$.
\end{proof}


In order to apply the above proposition we will use the following characterization for pushout diagrams.

\begin{lemma}\label{long-pushout}
In a commutative diagram 
\[\xymatrix{
0\ar[r] & V\ar[r]^{\iota} \ar[d]^{\alpha} & L_{-1} \ar[r]^\rho\ar[d]^{\delta} & L_0\ar[r] \ar[d] & T_1 \ar[r]\ar@{=}[d] & 0 \\
0\ar[r] & U\ar[r]^{\upsilon} & X \ar[r]^{\xi} & T_0\ar[r]  & T_1 \ar[r] & 0 
}\] the middle square is a pushout if and only if $\alpha$ is an epimorphism.  
\end{lemma}

\begin{proof}
Suppose that the middle square is a pushout. 
We consider $\pi:U\to U/\Ima(\alpha)$ the canonical epimorphism,
and $\mu:U/\Ima(\alpha)\to E$ is the embedding of $U/\Ima(\alpha)$ into its injective envelope. 
There exists a homomorphism $\nu:X\to E$ such that $\nu\upsilon=\mu\pi$, hence $\nu\delta\iota=0$. Then 
$\nu\upsilon$ factorizes through $\rho$. Since the middle square is a pushout, $\nu$ factorizes through $\xi$. It follows that 
$\mu\pi=\nu\upsilon=0$. Since $\mu$ is monic, we obtain $\pi=0$, hence $\alpha$ is an epimorphism.

Conversely, if $\alpha$ is an epimorphism and we have two homomorphisms $\beta_1:X\to Y$ and $\beta_2:L_0\to Y$ such that $\beta_1\delta=\beta_2\rho$
then $\beta_1\upsilon\alpha=0$, hence $\beta_1\upsilon=0$. It follows that there exists a unique homomorphism 
$\overline{\beta}:\Ima(\xi)\to Y$ such that $\beta_1=\overline{\beta}\upsilon$. 

Let $\overline{\delta}:\Ima(\rho)\to \Ima(\xi)$ be the homomorphism induced by $\delta$. If $\iota_\rho:\Ima(\rho)\to L_0$ and $\iota_\xi:\Ima(\xi)\to T_0$ are the canonical inclusions, then $\overline{\beta}\,\overline{\delta}=\beta_2\iota_{\rho}$. Since the first square in the commutative diagram
\[\xymatrix{
0\ar[r] & \Ima(\rho) \ar[r]^{\iota_\rho}\ar[d]^{\overline{\delta}} & L_0\ar[r] \ar[d] & T_1 \ar[r]\ar@{=}[d] & 0 \\
0\ar[r] & \Ima(\xi) \ar[r]^{\iota_\xi} & T_0\ar[r]  & T_1 \ar[r] & 0 
}\]
is a pushout, there exists a unique homomorphism $\beta^*:T_0\to Y$ such that $\overline{\beta}$ (hence $\beta_1$) and $\beta_2$ factorize through $\beta^*$, and the proof is complete.
\end{proof}

Let $Y$ be a submodule of a module $P$, and consider a canonical projection $\pi:P \to P/Y$.
Recall that $Y\ll P$ means that $Y$ is a \textsl{superfluous} submodule of a module $P$, i.e. that $\varphi$ is an epimorphism for every $\varphi\in\Hom(M,P)$ such that $\pi\varphi$ is an epimorphism.

We will need the following easy observation:

\begin{lemma}\label{sf} Let $X,P,T$ be modules over a ring $R$ such that $X\ll P$ and  $\alpha\in\Hom(P,T)$.
Then $\alpha(X)\ll\alpha(P)$. If, furthermore, $\alpha(X)=\alpha(P)$ then $\alpha=0$.
\end{lemma}

Now we are ready to characterize torsion classes generated by silting modules over right perfect rings. 

\begin{theorem}\label{perfect-silting} Let $R$ be a right perfect ring and $\CT\subseteq \Modr R$ a torsion class.
The following are equivalent:
\begin{enumerate}[{\rm (1)}]
\item $\CT=\Gen(T)$ for a silting module $T$;
\item There exists a $\CT$-preenvelope $\epsilon:R\to M$ such that $M\in \CT\cap {^\perp\CT}$. 
\end{enumerate}

In these conditions, if $K=\Coker(\epsilon)$ then $M\oplus K$ is a silting module, and $\CT=\Gen(M\oplus K)$.
\end{theorem}

\begin{proof}
(1)$\Rightarrow$(2)  This is a consequence of Proposition \ref{silting-classes} (see also \cite[Proposition 3.11]{AMV:2015}). 

%

(2)$\Rightarrow$(1) We consider the exact sequence $0\to U\overset{\iota_U}\to R\overset{\epsilon}\to M\overset{\rho}\to K\to 0$
where we consider $\iota_U$ as the inclusion map.  
If $\overline{\epsilon}:R/U\to M$ is the homomorphism induced by $\epsilon$ then for every $T\in \CT$ the homomorphism $\Hom(\overline{\epsilon},T)$ is an epimorphism. Since $M\in {^\perp \CT}$ we obtain $K\in {^\perp \CT}$.


For an epimorphism $\gamma:P_0\to M$ with $P_0$ projective, we have a commutative diagram \[\xymatrix{
& P_{-1} \ar[r]^{\sigma}\ar[d]^{\pi_{\sigma}}  & P_0\ar[r] \ar@{=}[d] & K \ar[r]\ar@{=}[d] & 0\\
0 \ar[r]& Z \ar[r]^{\overline{\sigma}}\ar[d]^{\overline{\delta}}  & P_0\ar[r] \ar[d]^{\gamma} & K \ar[r]\ar@{=}[d] & 0 \\ 
0\ar[r] &R/U \ar[r]^{\overline{\epsilon}} & M\ar[r]  & K \ar[r] & 0, 
}\]
where $\overline{\sigma}$ and $\overline{\epsilon}$ are the canonical homomorphisms induced by $\sigma$ and $\epsilon$, respectively, and 
$\pi_\sigma:P_{-1}\to Z$ is the projective cover of $Z$.

Since $P_{-1}$ is projective we can construct a commutative diagram
\[\xymatrix{
0 \ar[r]&X \ar[r]^{\iota_X}\ar[d]^{\upsilon}& P_{-1}\ar[r]^{\sigma}\ar[d]^{\delta} & P_0\ar[r] \ar[d]^{\gamma} & K \ar[r]\ar@{=}[d] & 0 \\ 
0\ar[r]&U\ar[r]^{\iota_U} &R \ar[r]^{\epsilon} & M\ar[r]  & K \ar[r] & 0 
}\]
such that $\pi_U\delta=\overline{\delta}\pi_\sigma$, where $\pi_U:R\to R/U$ is the canonical projection and 
$\iota_X$ is the inclusion map. Moreover,  note that $X=\ker\sigma=\ker\pi_\sigma$, hence $X\ll P_{-1}$.

We will show that $\Hom(\sigma, T)$ is  onto for all $T\in\mathcal T$. Similar techniques were also used in \cite{trl-tau}.
Let us consider the short exact sequence $$0\to Z\overset{\overline{\sigma}}\to P_0\to K\to 0,$$ and note that for every $T\in \mathcal T$
we have  a short exact sequence
\[
(*)\ \ \  0\to\Hom(K,T)\to\Hom(P_0,T)\overset{\Hom(\overline{\sigma}, T)}\longrightarrow\Hom(Z,T)\to 0
\]
since $\Ext^1(K,T)=0$. Fix an arbitrary $T\in \mathcal T$ and an arbitrary $\varphi\in \Hom(P_{-1},T)$.
Let us denote by $\pi_X:P_{-1}\to P_{-1}/X$ and $\pi_T:T\to T/\varphi(X)$
the canonical projections. Then we can find a homomorphism $\overline{\varphi}\in\Hom(P_{-1}/X,T/\varphi(X))$
which satisfies $\overline{\varphi}\pi_X=\pi_T\varphi$.
As $T/\varphi(X)\in\mathcal T$, there exists $\overline{\psi}\in\Hom(P_0,T/\varphi(X))$ for which $\overline{\psi}\sigma =\overline{\varphi}\pi_X$ by the
exactness of $(*)$.
Since $P_0$ is projective and $\pi_T$ is an epimorphism, $\overline{\psi}$ factorizes through $\pi_T$, i.e. there exists  $\psi\in\Hom(P_0,T)$
such that $\pi_T\psi=\overline{\psi}$.
Hence
\[
\pi_T\psi\sigma = \overline{\psi}\sigma =\overline{\varphi}\pi_X = \pi_T\varphi.
\]
Put $\alpha:=\varphi-\psi\sigma\in \Hom(P_{-1},T)$. From $\pi_T\alpha=0$ we have
$\alpha(P_{-1})\subseteq \varphi(X)$. Furthermore, $\alpha|_{X}=\varphi|_{X}$
since $\psi\sigma(X)=0$, which implies 
that $\alpha(P_{-1})\subseteq \alpha(X)$. By Lemma \ref{sf} we obtain $\alpha=0$, so $T\in \mathcal D_\sigma$ .

Using the pushout of $\sigma$ and $\delta$ we obtain a commutative diagram 
\[\xymatrix{
0 \ar[r]&X \ar[r]\ar[d]^{\upsilon'}& P_{-1}\ar[r]^{\sigma}\ar[d]^{\delta} \pushoutcorner & P_0\ar[r] \ar[d]^{\gamma'} & K \ar[r]\ar@{=}[d] & 0 \\ 
0\ar[r]&V\ar[r]\ar[d]^{\overline{\upsilon}} &R \ar[r]^{\epsilon'}\ar@{=}[d] & L\ar[r]^{\rho'}\ar[d]^{\overline{\gamma}}  & K \ar[r]\ar@{=}[d] & 0\\
0\ar[r]&U\ar[r] &R \ar[r]^{\epsilon} & M\ar[r]^{\rho}  & K \ar[r] & 0, 
}\]
such that $\overline{\upsilon}\upsilon'=\upsilon$ and $\gamma=\overline{\gamma}\gamma'$. In order to simplify the presentation, let us remark that $\upsilon'$ is surjective and $\overline{\upsilon}$ is injective, hence $V$ can be identified with the image of $\upsilon$. In this case  the equality $\overline{\upsilon}\upsilon'=\upsilon$ represents the canonical decomposition of $\upsilon$ through its image.

We will prove that $\overline{\gamma}$ is a $\mathcal{T}$-preenvelope for $L$.  For every $T\in\mathcal{T}$ and every homomorphism $\alpha:L\to T$ 
there exists $\overline{\alpha}:M\to T$ such that $\alpha\epsilon'=\overline{\alpha}\epsilon$. 
Then $(\overline{\alpha}\,\overline{\gamma}-\alpha)\epsilon'=0$, hence there exists $\beta:K\to T$ such that 
$\overline{\alpha}\,\overline{\gamma}-\alpha=\beta\rho'=\beta\rho\overline{\gamma}$. Then 
$\alpha=(\overline{\alpha}-\beta\rho)\overline{\gamma}$, and the claim is proved. 

Therefore, since $M\in{^\perp \mathcal{T}}$, applying the functors $\Hom(-,T)$ with $T\in\CT$ to the exact sequence $0\to \Ker(\overline{\gamma})\to L\to M\to 0$ it follows that $\Hom(\Ker(\overline{\gamma}),T)=0$  for all $T\in \mathcal{T}$. 

We split the bottom rectangle in the previous commutative diagram in two commutative diagrams with short exact sequences,
\[\xymatrix{
0\ar[r]&V\ar[r]\ar[d]^{\overline{\upsilon}} &R \ar[r]^{\epsilon'}\ar@{=}[d] & \Ima(\epsilon') \ar[r]\ar[d]^{\zeta}  &  0\\
0\ar[r]&U\ar[r] &R \ar[r]^{\epsilon} & \Ima(\epsilon)\ar[r]  &  0, 
}\]
and 
\[\xymatrix{
0\ar[r]&\Ima(\epsilon')\ar[d]^{\zeta}\ar[r] & L\ar[r]^{\rho'}\ar[d]^{\overline{\gamma}}  & K \ar[r]\ar@{=}[d] & 0\\
0\ar[r]&\Ima(\epsilon)\ar[r] & M\ar[r]^{\rho}  & K \ar[r] & 0, 
}\]
where $\zeta$ can be identified to the canonical surjection $R/V\to R/U$. 
Applying Ker-Coker Lemma, we observe that $C=\Coker(\overline{\upsilon})\cong \Ker(\overline{\gamma})$.

If $\pi:P\to C$ is a  projective cover for $C$ and $\alpha:P\to T$ is a homomorphism with $T\in\mathcal{T}$ then 
the induced homomorphism $\overline{\alpha}:P/\Ker(\pi)\to T/\alpha(\Ker(\pi))$ defined by the rule   
$$\overline{\alpha}(x+\Ker(\pi))=\alpha(x)+\alpha(\Ker(\pi)),$$ is zero. Therefore $\alpha(P)=\alpha(\Ker(\pi))$. Since $\Ker(\pi)$ is superfluous, it follows that $\alpha=0$, hence $\Hom(P,\mathcal{T})=0$. 

We lift $\pi$ to a homomorphism $\overline{\pi}:P\to U$. By Lemma \ref{long-pushout}, $\upsilon'$ is an epimorphism, and it is easy to see that $\Ima(\upsilon)+\Ima(\overline{\pi})=\Ima(\overline{\upsilon})+\Ima(\overline{\pi})=U$. Therefore, the canonical map $(\overline{\pi},\upsilon):P\oplus X\to U$ induced by $\overline{\pi}$ and $\upsilon$ is an epimorphism.

 Now we construct the commutative diagram 
\[\xymatrix{
0 \ar[r]&P\oplus X \ar[r]^{1_P\oplus \iota_X}\ar[d]^{(\overline{\pi},\upsilon)}& P\oplus P_{-1}\ar[r]^{(0,\sigma)}\ar[d]^{(\iota_U\overline{\pi},\delta)} 
& P_0\ar[r] \ar[d]^{\gamma} & K \ar[r]\ar@{=}[d] & 0 \\ 
0\ar[r]&U\ar[r]^{\iota_U} &R \ar[r]^{\epsilon} & M\ar[r]  & K \ar[r] & 0. 
}\]
Since $(0,\sigma)\in {^\square\mathcal{T}}$, it remains to apply Lemma \ref{long-pushout} and Proposition \ref{silting-classes} to complete the proof.
\end{proof}

The following class of examples, used in commutative case also in \cite{An-Hr}, shows that the implication (2)$\Rightarrow$(1) does not hold in general.

\begin{example}\label{example-semiperfect} Let $R$ be a semiperfect ring with non-zero idempotent Jacobson radical $J$, i.e. $J^2=J\ne0$.
Denote by $S_i$ simples and by $P_i$ the corresponding indecomposable projectives such that $\oplus_{i\le n} S_i=R/J$ and $S_i\cong P_i/P_iJ$. Since
idempotency of $J$ implies that extensions of semisimple modules by semisimple modules are semisimple as well, we can see that 
$\CT=\Gen(R/J)=\{\oplus_i S_i^{(\kappa_i)}|\ \kappa_i, i\le n \}$ 
is a torsion class and 
$\Ext^1(T,U)=0$ for each $T,U\in\CT$, hence $R/J\in \CT\cap {^\perp\CT}$. Furthermore, it is easy to verify that the natural projection $R\to R/J$ forms a $\CT$-envelope of $R$.
We will show that no generator $G=\oplus_i S_i^{(\kappa_i)}$ of $\CT$ is silting.

Consider an exact sequence $P_{-1}\overset{\sigma}\to P_0\overset{\rho}\to \oplus_i S_i^{(\kappa_i)}\to 0$. Since $\rho$ factorizes through the canonical projection
$\pi:\oplus_i P_i^{(\kappa_i)}\to \oplus_i S_i^{(\kappa_i)}$ we may suppose that $P_0=\oplus_i P_i^{(\kappa_i)}$ and that $\rho=\pi$. 
Note that $\Ima(\sigma)=\oplus_i P_i^{(\kappa_i)}J\ne 0$ because $\oplus_i S_i^{(\kappa_i)}$ generates $\CT$, which implies that $P_{-1}\ne 0$.
Since for every $T\in\CT$ and every homomorphism $\varphi\in \Hom(P_0,T)$ we have $\Ima(\sigma)=P_0J\subseteq\ker(\varphi)$, the composition $\varphi\sigma0$ is zero.
As $\Hom(\sigma,T)=0$ while $\Hom(P_{-1},T)\ne 0$ for all nonzero $T\in\CT$, we can conclude that $G$ is not silting.

Finally note that the class of semiperfect rings with non-zero idempotent Jacobson radical contains for example all valuation domains with infinitely
generated maximal ideals.
\end{example}

In the case of hereditary rings, silting torsion classes can be characterized by the existence of a special long exact sequence.

\begin{theorem}\label{hereditary-silting} Let $R$ be a right hereditary ring and $\CT\subseteq \Modr R$ a torsion class.
The following are equivalent:
\begin{enumerate}[{\rm (1)}]
\item $\CT=\Gen(T)$ for a silting module $T$;
\item There exists a $\CT$-preenvelope $\epsilon:R\to M$ such that $M\in \CT\cap {^\perp\CT}$. 
\item There exists an exact sequence $0\to U\to R\to M\to K\to 0$ such that $M\in \CT$, $U\in {^\circ \CT}$ and $K\in {^\perp \CT}$.
\end{enumerate}
\end{theorem}

\begin{proof}
(1)$\Rightarrow$(2) The argument is the same as in the proof of Theorem~\ref{perfect-silting}, i.e. we apply Proposition \ref{silting-classes}.

(2)$\Rightarrow$(3) 
As in the proof of Theorem \ref{perfect-silting} we obtain  $K\in {^\perp \CT}$. 
  
Since $\epsilon$ is a $\CT$-preenvelope, every homomorphism $R\to T$ with $T\in \CT$ factorizes through $R/U$. Therefore, for every $T\in \CT$ we have that $\Hom(\pi,T)$ is an isomorphism, where $\pi:R\to R/U$ is the canonical epimorphism. Then 
first natural homomorphism from the exact sequence 
$$0\to \Hom(R/U,T)\to \Hom(R,T)\to \Hom(U,T)\to \Ext^1(R/U,T)$$ is an isomorphism.  
Moreover, using the exact sequence $0\to R/U\to M\to K\to 0$, we obtain $\Ext^1(R/U,T)=0$ for all $T\in \CT$. 
Therefore $\Hom(U,T)=0$ for all $T\in \CT$.

(3)$\Rightarrow$(1)
Since $R$ is hereditary, there exists a projective resolution $$0\to P_{-1}\overset{\sigma}\to P_{0}\to K\to 0.$$ Using 
the hypothesis $K\in {^\perp \CT}$, it follows that $\sigma\in{^\square \CT}$. If $U=\Ker(\epsilon)$ we can construct, as in the proof of Theorem \ref{perfect-silting}, using the projectivity of $P_0$, a commutative diagram 
\[\xymatrix{
0\ar[r] & U\ar[r]\ar@{=}[d]& U\oplus P_{-1} \ar[r]^{(0,\sigma)}\ar[d]^{(\iota,\delta)}  & P_0\ar[r] \ar[d]^{\gamma} & K \ar[r]\ar@{=}[d] & 0\\
 0\ar[r]& U\ar[r] &R \ar[r]^{\epsilon} & M\ar[r]  & K \ar[r] & 0, 
}\]
where $\iota:U\to R$ is the inclusion map. Since $U$ is projective, by $U\in {^\circ \mathcal{T}}$ it follows that $(0,\sigma)\in {^\square \mathcal{T}}$. From Proposition \ref{silting-classes} we conclude that $\mathrm{Gen}(T)=\mathcal{T}$ is a silting class.
\end{proof}

The following example shows that the condition $U\in{^\circ \mathcal{T}}$ is essential in the proof of (3)$\Rightarrow$(1).

\begin{example}
Let $\mathcal{T}=\Gen(\mathbb{Z}(p^\infty))$ in the category $\Modr \mathbb{Z}$ for a prime number $p$. It is easy to see that $\mathcal{T}$ is the class of all injective abelian $p$-groups, so it is a torsion class and for every $K\in \mathcal{T}$ we have $K\in {^\perp\mathcal{T}}$. Therefore, for every homomorphism 
$\epsilon:\mathbb{Z}\to T$ with $T\in\mathcal{T}$ we have $\Coker(\epsilon)\in {^\perp\mathcal{T}}$ and $\Ker(\epsilon)\cong \mathbb{Z}\notin {^\circ\mathcal{T}}$. Moreover, $\mathcal{T}$ is not generated by a silting module since it is not closed with respect to direct products, so it is not definable.  
\end{example}

On the other side, in the case of perfect rings there exists a torsion class $\mathcal{T}$ generated by a silting module such that 
$U\notin {^\circ \mathcal{T}}$.

\begin{example}\label{example-U}
We consider, as in \cite[Example 4.1]{AMV:2015} the $k$-algebra 
$$R=kQ/(\alpha\beta\alpha,\beta\alpha\beta),$$ where $Q$ is the quiver 
\[\xymatrix{1 \ar@/^/[r]^{\alpha}\ar@{<-}@/_/[r]_{\beta}& 2}.\] 
If $S_1$ and $S_2$ are the simple $R$-modules, respectively $P_1$ and $P_2$ are the corresponding projectives, then $M=S_1\oplus P_1\oplus P_1$ is a silting module. A $\Gen(M)$-preenvelope for $R$ is given by 
$$0\to U\to P_1\oplus P_2\overset{1_{P_1}\oplus \varphi}\longrightarrow P_1\oplus P_1\to S_1\to 0,$$
where $P_2\overset{\varphi}\to P_1\to S_1\to 0$ is the minimal projective presentation for $S_1$. It is not hard to see that $\Hom(U,S_1)\cong \Ext^1(S_2,S_1)\neq 0$.  
\end{example}

Let us recall that an infinitely generated module $T$ is (quasi-)tilting if $\Gen(T)=T^\perp$ ($\Pres(T)=\Gen(T)\subseteq T^\perp$). The class of silting modules is an intermediate class between the class of tilting modules and that of finendo quasi-tilting modules. Using a theorem of Wei, \cite{Wei-14}, it is proved in \cite[Proposition 3.15]{AMV:2015} that in the case of finitely generated modules over finitely dimensional algebras the silting finitely generated modules coincide to 
(finendo) quasi-tilting modules. In the case of hereditary or right perfect rings we obtain a similar result:

\begin{corollary}
Let $R$ be a right hereditary or right perfect ring. If $Q$ is a finendo quasi-tilting module then there exists a silting module $T$ such that $\Add(Q)=\Add(T)$. 

Consequently, the following are equivalent for a torsion class $\CT\subseteq \Modr R$:
\begin{enumerate}[{\rm (1)}]
\item $\CT=\Gen(T)$ for a silting module $T$;
\item $\CT=\Gen(T)$ for a finendo quasi-tilting module $T$.
\end{enumerate}
\end{corollary}

\begin{proof}
Let us recall from \cite[Proposition3.2 and Theorem 3.4]{AMV:2015} that $Q$ is finendo quasi-tilting if and only if there exist an exact sequence $$R\overset{\alpha}\to Q_0\to Q_1\to 0$$ such that $\alpha$ is a $\Gen(Q)$-preenvelope, $Q_0,Q_1\in\Add(Q)$ and $Q_1\in ^\perp \Gen(Q)$. From the proof of Theorem \ref{perfect-silting} and Theorem \ref{hereditary-silting} it follows that $T=Q_0\oplus Q_1$ is a silting module. Not it is easy to see that $\Add(Q)=\Add(T)$ and $\Gen(Q)=\Gen(T)$.   
%
%
\end{proof}

Using Example \ref{example-semiperfect} in the same manner as it is used in \cite{An-Hr} we observe that the equivalence from the above corollary is not true for general rings.

\begin{example}
Let $R$ be a valuation domain such that its maximal ideal is idempotent. Then the simple module $S$ is $\Modr R$ is finendo quasi-tilting. But, we proved in Example \ref{example-semiperfect} that $\Gen(S)$ is not generated by a silting module.   
\end{example}


 

We recall from \cite{AMV2} that a silting module $T$ is \textsl{minimal} if there exists a $\Gen(T)$-envelope for the regular module $R$. 
In order to apply the above results to minimal silting modules we need a lemma whose proof is included for reader's convenience.

\begin{lemma}\label{perp}
Let $\CT$ be a class of modules. If $\epsilon:R\to M$ is an $\CT$-envelope then every epimorphism $\alpha:N\to M$ with $N\in \CT$ splits.  
Consequently, if $\CT$ is a class closed under extensions and $\epsilon:R\to M$ is a $\CT$-envelope then $M\in {^\perp\CT}$.
\end{lemma}

\begin{proof}
Since $\alpha$ is an epimorphism, there exists $\gamma:R\to N$ such that $\epsilon=\alpha\gamma$. Then there exists 
exists $\beta:M\to N$ such that $\beta\epsilon=\gamma$. It follows that $\alpha\beta\epsilon=\epsilon$. Since $\epsilon$ is a $\CT$-envelope, it follows that $\alpha\beta$ is an automorphism, hence $\alpha$ splits. 
%

The last statement is now obvious since in every short exact sequence $0\to T\to N\to M\to 0$, with $T\in\CT$, we have $N\in \CT$.
\end{proof}

\begin{corollary}\label{env-her}
The following are equivalent for a torsion class $\CT$ of modules over a right hereditary or right perfect ring $R$:
\begin{enumerate}[{\rm (1)}]
\item $\CT=\Gen(T)$ for a minimal silting module $T$; 
\item There exists a $\CT$-envelope $\epsilon:R\to M$.
\end{enumerate}

In particular, all enveloping torsion classes over hereditary or right perfect rings are generated by silting modules.
\end{corollary}


Moreover, a half of Salce's Lemma \cite[Lemma 5.20]{Gobel_Trlifaj:2006} is valid for silting modules:

\begin{proposition}\label{half-salce}
Let $T$ be a silting module. If $\CT=\Gen(T)$ then for every $R$-module $X$ there exists a short exact sequence  
$$0\to L\to U \overset{\upsilon}\to X\to 0$$ such that $\upsilon$ is a ${^\diamond\CT}$-precover for $X$ and $L\in\CT$.

Consequently, ${^\perp\CT}$ is a special precovering class. 
\end{proposition}

\begin{proof}
If $X$ is an $R$-module, we consider a pushout diagram 
$$\xymatrix{0\ar[r]& Y\ar[r]^\upsilon\ar[d]^{\alpha} & P\ar[r]\ar[d]^\beta & X\ar[r]\ar@{=}[d] & 0 \\
0\ar[r] &L\ar[r]^\gamma & U\ar[r] & X\ar[r] & 0,
}$$ where $P$ is a projective module and $\alpha$ is a $\CT$-preenvelope for $Y$ obtained as a pushout 
\[\xymatrix{
 P_{-1}\ar[r]^{\zeta} \ar[d]^{\upsilon'} & P_0 \ar[r]\ar[d] & Z\ar@{=}[d] \ar[r] & 0\\
  Y\ar[r]^{\alpha}  & L \ar[r] & Z\ar[r] & 0 
}\] for some $\zeta\in{^\square\CT}$. Then we have a pushout square 
$$\xymatrix{ P_{-1}\ar[r]^{\upsilon\upsilon'}\ar[d]^\zeta & P\ar[d]^{\beta}  \\
 P_0\ar[r]^{\gamma\gamma'} & U,
}$$ hence $U$ is the cokernel of the homomorphism $\delta:P_{-1}\to P_0\oplus P$ induced by $\upsilon\upsilon'$ and $\zeta$. 
Since every homomorphism $f:P_{-1}\to T$ with $T\in \CT$ can be writen as $f=g\zeta$ for some $g:P_0\to T$, it follows that 
$f=g'\delta$, where $g':P_{0}\oplus P\to T$ is defined by $g'_{|P_{0}}=g$ and $g'_{|P}=0$. Then $\delta\in{^\square\CT}$, so $U\in{^\diamond\CT}$. 

Now, for every $V \in {^\diamond\CT}$ we have $\CT\subseteq {V^\perp}$, and it follows that
$\gamma$ is a  ${^\diamond\CT}$-precover for $X$.

The last statement follows from the inclusion ${^\diamond\CT}\subseteq {^\perp \CT}$.
\end{proof}



\section{Cosilting classes}

For the dual results, let us recall from \cite{Br-Po} that we can associate to every homomorphism $\sigma:Q_{0}\to Q_1$  between injective modules the class 
$$\CB_{\sigma} = \{ X\in\Modr R \mid \Hom_{R}(X,\sigma) \text{ is an epimorphism}\},$$ and  a module $T$ is \textsl{partial cosilting} if there exists an injective presentation $$0\to T\to Q_{0}\overset{\sigma}\to Q_1$$ such that 
$\CB_\sigma$ is a torsion-free class and $T\in \CB_\sigma$. Then $\Cogen(T)\subseteq \CB_\sigma\subseteq {^\perp T}$.
If  $\CB_\sigma=\Cogen(T)$ then $T$ is called \textsl{cosilting}.

Let $\CI$ be the class of all injective modules, and $\CI^\to$ the class of all homomorphisms between injective modules. 
If $\CF$ is a class of right $R$-modules then we associate to $\CF$ the following classes 
\begin{itemize}
\item $\CF^\square=\{\sigma:S_0\to S_1\mid \sigma\in\CI^\to \textrm{, and }\CF\subseteq \CB_\sigma\},$ and

\item $\CF^\diamond=\{\Ker(\sigma)\mid \sigma\in \CF^\square\}$.
\end{itemize}

In order to dualize Theorem~\ref{perfect-silting} and Corollary~\ref{env-her} let us formulate dual versions of Propositions~\ref{silting-classes}.

\begin{proposition}\label{cosilting-classes}
The following are equivalent for a class $\CF$ of $R$-modules:
\begin{enumerate}[{\rm (1)}]
\item There exists a cosilting module $T$ such that $\CF=\Cogen(T)$;

\item \begin{enumerate}[{\rm (a)}]
\item $\CF$ is a torsion-free class,
\item If $E$ is a fixed injective cogenerator for $\Modr R$ then there exists an $\CF$-precovering $\epsilon:M\to E$ 
which can be obtained as a pullback 
\[\xymatrix{
0 \ar[r] & K\ar[r] \ar@{=}[d] & M \ar[r]^{\epsilon}\ar[d] & E\ar[d]^{\nu} \\
0 \ar[r] & K\ar[r]  & Q'_0 \ar[r]^{\zeta'} & Q'_1 
}\]
such that $\zeta'\in\CF^\square$.
\end{enumerate}
\item \begin{enumerate}[{\rm (a)}]
\item $\CF$ is a torsion-free class,
\item for every $R$-module $X$ there exists an $\CF$-precovering $\alpha:M\to X$ which can be obtained as a pullback 
\[\xymatrix{
0 \ar[r] & S\ar[r] \ar@{=}[d] & M \ar[r]^{\alpha}\ar[d] & X\ar[d] \\
0 \ar[r] & S\ar[r]  & S_0 \ar[r]^{\sigma} & S_1 
}\]
such that $\sigma\in\CF^\square$.
\end{enumerate}
\end{enumerate}

If we have a diagram as in {\rm (2)} then $K\oplus M$ is a cosilting modules and $\CF=\Cogen(K\oplus M)$.
\end{proposition}

If $Y$ is a submodule of a module $P$ with the canonical embedding $\nu: Y\to P$, then $Y$ is an essential submodule of 
$P$, $Y\trianglelefteq P$, if an arbitrary homomorphism $\varphi\in\Hom(P,N)$ is a monomorphism whenever $\varphi\nu$ is a monomorphism. 

\begin{lemma}\label{sf2} Let $Y,Q,F$ be modules over a ring $R$ such that  $Y\trianglelefteq Q$ and $\alpha\in\Hom(F,Q)$.
Then  $\beta(F)\cap Y\trianglelefteq \beta(F)$.  If, furthermore, $\beta(F)\cap Y=0$, then $\beta=0$.
\end{lemma}

We will also use the dual of Lemma \ref{perp}. 
\begin{lemma}\label{dual-perp}
Suppose that $\CF$ is a class of modules and $\epsilon:M\to E$  is an $\CF$-cover of an injective module $E$. Then every monomorphism $\alpha:M\to N$ with $N\in\CF$ splits.

Therefore, if $\CF$ is a class closed under extensions and $\epsilon:M\to E$  is an $\CF$-cover of an injective module $E$ then $M\in {\CF^\perp}$.
\end{lemma}

As in the (co)tilting theory,  we obtain the following:

\begin{lemma}\label{corr-covering}
If $T$ is a cosilting module then $\Cogen(T)$ is a covering class. 
\end{lemma}

\begin{proof}
It is proved in \cite[Corollary 4.8]{Br-Po} that $\Cogen(T)$ is closed under direct limits.  
Using \cite[Theorem 5.31]{Gobel_Trlifaj:2006} we conclude that $\Cogen(T)$ is a covering class.
\end{proof}

Since every module has an injective envelope over an arbitrary ring, application of dual techniques to that applied in the silting case 
gives us the dual result to Corollary~\ref{env-her} that cosilting classes are exactly that torsion-free classes which are covers over general rings.

\begin{theorem}\label{cosilting-hereditary}
Let $R$ be a ring and $E$ a fixed injective cogenerator for $\Modr R$. If $\CF$ is a torsion-free class in $\Modr R$, the following are equivalent:

\begin{enumerate}[{\rm (1)}]
\item $\CF=\Cogen(T)$ for a cosilting module $T$;
\item $\CF$ is a covering class;
\item there exists an $\CF$-cover $\epsilon:M\to E$;
\item There exists an $\CF$-precover $\epsilon:M\to E$ such that $M\in \CF\cap {\CF^\perp}$.
\end{enumerate}

Moreover, if $R$ is hereditary, then the above conditions are equivalent to:

\noindent{\rm (5)} There exists an exact sequence $0\to K\to M\to E\to V\to 0$ such that $M\in \CF$, $V\in { \CF^\circ}$ and $K\in { \CF^\perp}$.

In these conditions, if $K=\Ker(\epsilon)$ then $M\oplus K$ is a cosilting module and $\CF=\Cogen(M\oplus K)$.
\end{theorem}

\begin{proof} 
The implication (1)$\Rightarrow$(2) follows from Lemma \ref{corr-covering} and (2)$\Rightarrow$(3) is trivial. 
The other implications in a cyclic proof are dual to those presented in the silting case. 
\end{proof}

Let us note that the equivalence (1)$\Leftrightarrow$(2) was proved independently by Zhang and Wei, cf. \cite[Theorem 3.5]{ZW} and \cite[Theorem 4.18]{ZW2}.

In the following example we will see that the property $V\in { \CF^\circ}$ cannot be deduced if $R$ is not hereditary.

\begin{example}
Let $R$ be the ring used in Example \ref{example-U}. If $(-)^d$ is the standard duality between right and left finitely presented modules, 
and $M$ is the silting module used in   Example \ref{example-U} then we can use the proof of \cite[Corollary 3.7]{Br-Po} to see that 
$M^d$ is a cosilting module and 
$$0\to S_1^d\to P_1^d\oplus P_1^d \overset{1_{P_1^d}\oplus \varphi^d}\longrightarrow R^d\to U^d\to 0$$
is the exact sequence induced by the $\Cogen(M^d)$-cover $1_{P_1^d}\oplus \varphi^d$ for $R^d$ such that 
$\Coker(1_{P_1^d}\oplus \varphi^d)=U^d$ is not in $\Cogen(M^d)^\circ$.
\end{example}

We have also the dual of Proposition \ref{half-salce}.

\begin{proposition}\label{half-salce-co}
Let $T$ be a cosilting module. If $\CF=\Cogen(T)$ 
then for every $R$-module $X$ there exists a short exact sequence  
$$0\to X\overset{\upsilon}\to U \to F\to 0$$ such that $\upsilon$ is a $\CF^\diamond$-preenveloping for $X$ and $F\in\CF$.
\end{proposition}

\begin{corollary}
Let $\CF=\Cogen(T)$ for a cosilting module $T$. Then the class $$\CF^\perp=\{X\in\Modr R\mid \Ext_R^1(F,X)=0\text{ for all }F\in\CF\}$$ is an enveloping class.
\end{corollary}

\begin{proof}
Since $\CF^\diamond\subseteq \CF^\perp$, it follows that every $\CF^\diamond$-preenvelope constructed in the previous 
proposition is a special $\CF^\perp$-preenvelope. Therefore, it is enough to apply \cite[Theorem 5.27]{Gobel_Trlifaj:2006} and  
\cite[Corollary 4.8]{Br-Po} to obtain the conclusion.
\end{proof}


\end{document}